\documentclass[a4paper]{amsart}

\usepackage{amssymb}
\usepackage{eucal}
\usepackage{xparse}
\usepackage{tikz-cd}
\usepackage{enumitem}
\usepackage{mathscinet}
\usepackage{hyperref}

\newcommand{\kk}{k}
\newcommand{\ZZ}{\mathbb{Z}}
\newcommand{\op}{{\mathrm{op}}}
\usepackage{mathtools}

\newcommand{\GG}[1]{} 

\newcommand{\category}[1]{\mathcal{#1}}

\newcommand{\noloc}{\rotatebox[origin=c]{180}{$\colon$}}

\newcommand{\A}{\category{A}}%
\newcommand{\B}{\category{B}}%
\newcommand{\C}{\category{C}}%
\newcommand{\F}{\category{F}}%
\newcommand{\G}{\category{G}}%
\renewcommand{\S}{\category{S}}%
\newcommand{\T}{\category{T}}%

\NewDocumentCommand{\DerCat}{O{}D<>{}m}{\operatorname{D}^{#1}_{\mathrm{#2}}(#3)}
\NewDocumentCommand{\dgDerCat}{O{}D<>{}m}{\operatorname{D}^{#1}_{\mathrm{#2}}(#3)_{\mathrm{dg}}}
\NewDocumentCommand{\K}{O{}D<>{}m}{\operatorname{K}^{#1}_{\mathrm{#2}}(#3)}

\NewDocumentCommand{\Mod}{m}{\operatorname{Mod}(#1)}
\NewDocumentCommand{\Ch}{m}{\operatorname{Ch}(#1)}

\NewDocumentCommand{\stableCat}{mO{\S}}{\underline{#1}_{#2}}

\NewDocumentCommand{\HC}{O{\bullet}m}{\operatorname{C}^{#1}(#2)}

\RenewDocumentCommand{\H}{O{\bullet}m}{\operatorname{H}^{#1}(#2)}

\NewDocumentCommand{\proj}{m}{\operatorname{proj}(#1)}

\NewDocumentCommand{\mmod}{O{}m}{\operatorname{mod}_{#1}(#2)}
\NewDocumentCommand{\pcoh}{m}{\operatorname{pcoh}(#1)}

\NewDocumentCommand{\Hom}{O{}mm}{\operatorname{Hom}_{#1}(#2,#3)}
\NewDocumentCommand{\Ext}{O{}mmO{\bullet}}{\operatorname{Ext}_{#1}^{#4}(#2,#3)}
\NewDocumentCommand{\End}{O{}m}{\operatorname{End}_{#1}(#2)}
\DeclareMathOperator{\colim}{\operatorname{colim}}

\NewDocumentCommand{\Lotimes}{O{}}{\otimes^{\mathbf{L}}_{#1}}
\NewDocumentCommand{\RHom}{O{}mm}{\mathbf{R}\!\operatorname{Hom}_{#1}(#2,#3)}
\NewDocumentCommand{\REnd}{O{}m}{\mathbf{R}\!\operatorname{End}_{#1}(#2)}
\NewDocumentCommand{\LFun}{O{}mm}{\operatorname{LFun}_{#1}(#2,#3)}

\NewDocumentCommand{\HH}{O{\bullet}m}{\operatorname{HH}^{#1}(#2)}

\NewDocumentCommand{\coh}{O{}m}{\operatorname{coh}_{#1}(#2)}
\NewDocumentCommand{\PP}{O{n}}{\mathbb{P}^{#1}}
\NewDocumentCommand{\set}{mo}{\left\{#1\IfValueT{#2}{\,\middle|\,#2}\right\}}

\usepackage[%
backend=bibtex,%
style=alphabetic,%
natbib=false,%
isbn=false, 
doi=false, 
url=false, 
eprint=true, 
giveninits=true, 
backref=false, 
maxbibnames=99,%
]{biblatex}%

\AtEveryBibitem{%
  \ifentrytype{unpublished}{}{\clearfield{note}}%
  }%
  \renewbibmacro{in:}{} 

  \usepackage{amsthm}%
  \usepackage{thmtools}%
  \usepackage{cleveref}%

  \numberwithin{equation}{section}%

  \declaretheorem[style=theorem,sibling=equation]{lemma}%
  \declaretheorem[style=theorem,sibling=equation]{proposition}%

  \declaretheorem[style=theorem,sibling=equation]{conjecture}%
  \declaretheorem[style=theorem,sibling=equation]{theorem}%
  \declaretheorem[style=definition,sibling=equation]{definition}%
  \declaretheorem[style=definition,numbered=no,name=Definition]{definition*}%
  \declaretheorem[style=definition,sibling=equation]{notation}%
  \declaretheorem[style=definition,numbered=no]{acknowledgements}%
  \declaretheorem[style=definition,numbered=no]{conventions}%
  \declaretheorem[style=remark,sibling=equation]{remark}%
  \title[Rickard's derived Morita theory]{Rickard's derived Morita theory: \\ Review and Outlook}

  \author[G.~Jasso]{Gustavo Jasso}%
  \author[H.~Krause]{Henning Krause} \author[S.~Schroll]{Sibylle Schroll}%

  \address[Jasso, Schroll]{%
    Mathematisches Institut, %
    Universität zu Köln, %
    Weyertal 86-90, %
    D-50931 Köln, %
    Germany}%
  \email{gjasso@math.uni-koeln.de}%
  \urladdr{https://gustavo.jasso.info}%

  \email{schroll@math.uni-koeln.de}%
  \urladdr{https://sites.google.com/site/sibylleschroll/}%

  \address[Krause]{ Fakult{\"a}t f{\"u}r Mathematik, %
    Universit{\"a}t Bielefeld, %
    D-33501 Bielefeld, %
    Germany }

  \email{hkrause@math.uni-bielefeld.de}%
  \urladdr{https://www.math.uni-bielefeld.de/~hkrause/}%

\begin{document}


  \begin{abstract}
    We survey the main results in Jeremy Rickard's seminal papers `Morita theory for derived categories' and `Derived equivalences and derived functors'. These papers catalysed the later development of the Morita theory of (enhanced) compactly generated triangulated categories by Keller in the algebraic setting and by Schwede and Shipley in the topological setting. We also discuss the role of Rickard's notion of splendid equivalence in the context of Broué's abelian defect group conjecture, and indicate an alternative proof of parts of Rickard's Derived Morita Theorem that leverages the notion of completion of a triangulated category.
  \end{abstract}

  \maketitle

  \tableofcontents



\newcommand{\calo}{{\mathcal{O}}}



\section{Introduction}

In a series of seminal papers~\cite{Ric89a,Ric89b,Ric91,Rickard96}
Rickard developed a Morita theory for derived categories.  Given a
pair of associative algebras, classical Morita theory provides
criteria for their module categories to be equivalent. 
The relevance of derived categories became apparent in
the 1980s, and it was Jeremy Rickard who provided an analogue of
Morita's theory in that context. This marked the beginning of an exciting
development that continues to give new insights, because derived
equivalences are often used as bridges between neighbouring fields in
mathematics. Such derived equivalences arise from tilting
objects, and tilting theory, as introduced by Brenner and Butler into the
representation theory of finite dimensional algebras, can be seen as
an inspiration for Rickard's theory.

This paper offers a survey and tries to put Rickard's work into a
broader context. We proceed in historical order from classical Morita theory via tilting to Rickard's derived Morita theory. 
The discussion emphasises the categorical aspects and moves on to Brou\'e’s abelian defect
group conjecture, including some of the main ideas in Rickard’s early work in this context. 
Then we focus on more recent aspects which arise from the triangulated structure of a derived category. For instance, there is a notion
of completion which provides a new interpretation of some of Rickard's results. Also, there are further classes of triangulated categories
beyond the derived categories of associative algebras, where a Morita theory has been developed. So we see that Rickard's work continues to inspire new developments far beyond its origins in representation theory.

\begin{acknowledgements}
    The authors are grateful to Bernhard Keller and Jeremy Rickard for numerous helpful comments on a previous version of this article. The authors thank Yankı Lekili for help with some references on Homological Mirror Symmetry and Radha Kessar for advice on results regarding  Brou\'e's abelian defect group conjecture. Furthermore, the authors thank Radha Kessar, Calvin Pfeifer and David Ploog for their comments and corrections.  The authors would also like to thank the anonymous referees whose thoughtful and detailed comments much improved this article. This work was supported by the Deutsche Forschungsgemeinschaft (SFB-TRR 358/1 2023 - 491392403).
\end{acknowledgements}

\begin{conventions}
Unless explicitly noted otherwise, we work over a commutative (unital) ring $k$, so that all algebras and categories that we consider are linear over $k$.  For an algebra $A$, we write $\Mod{A}$ for the abelian category of right $A$-modules and $\mmod{A}$ for its full subcategory of \emph{finitely presented} $A$-modules, that is $A$-modules that are cokernels of a morphism between finitely generated projective $A$-modules; notice that $\mmod{A}$ is a $k$-linear category with cokernels but in general it is not an abelian category. Finally, we write $\proj{A}$ for the idempotent-complete\footnote{An additive category is \emph{idempotent-complete} if every idempotent endomorphism induces a direct sum decomposition.} $k$-linear category of finitely generated projective $A$-modules.
\end{conventions}

\section{Rickard's derived Morita theory}

In this section we describe some of the main results in Rickard's papers~\cite{Ric89a,Ric89b,Ric91,Rickard96}. We begin with a cursory discussion of classical Morita theory emphasising its categorical aspects. This discussion is followed by a brief history of tilting theory, which serves as a bridge between classical Morita theory and Rickard's derived version thereof that we discuss afterwards. We conclude by recalling Brou{\'e}'s abelian defect group conjecture~\cite{Broue90} and some of the main ideas in Rickard's early work in this context.

\subsection{Classical Morita theory}

The following is one of the most basic questions in representation theory: Given two algebras $A$ and $B$, when are the module categories $\Mod{A}$ and $\Mod{B}$ equivalent as $\kk$-linear categories? This question was given a satisfactory answer by Morita in~\cite{Mor58}. Recall that an $A$-$B$-bimodule is a $k$-module $M$ with a left $A$-module structure and a right $B$-module structure that are compatible in the sense that
\[
    (a\cdot m)\cdot b = a\cdot (m\cdot b),\qquad a\in A,\ m\in M,\ b\in B.
\]

\begin{theorem}[Morita]
\label{thm:Morita}
    Let $A$ and $B$ be a pair of algebras. The following statements are equivalent.
    \begin{enumerate}
        \item \label{thm:Morita:it:eq} There exists a $\kk$-linear equivalence of categories
        \[
            \Mod{A}\stackrel{\sim}{\longrightarrow}\Mod{B}.
        \]
        \item\label{thm:Morita:it:std-eq} There exists an $A$-$B$-bimodule $M$ such that the $\kk$-linear functor
        \[
            -\otimes_A M\colon\Mod{A}\longrightarrow\Mod{B}
        \]
        is an equivalence.
        \item\label{thm:Morita:it:progen} There exists a finitely generated projective $B$-module $P$ with the following properties:
        \begin{enumerate}
            \item  There is an isomorphism of algebras
            \[
            A\cong\Hom[B]{P}{P}.
            \]
            \item The finitely generated projective $B$-module $P$ is a \emph{progenerator}, that is $P$ generates $\proj{B}$ as an idempotent complete additive category. Equivalently, the regular representation of $B$ is a direct summand of a finite direct sum of copies of $P$.
        \end{enumerate}
    \end{enumerate}
\end{theorem}

\begin{remark}
    The implication \eqref{thm:Morita:it:std-eq}$\Rightarrow$\eqref{thm:Morita:it:eq} in \Cref{thm:Morita} is clear, and the converse implication \eqref{thm:Morita:it:eq}$\Rightarrow$\eqref{thm:Morita:it:std-eq} follows readily from the classical Eilenberg--Watts Theorem~\cite{Eil60,Wat60}, which characterises the $\kk$-linear functors that are given by the tensor product with a bimodule. The implication \eqref{thm:Morita:it:progen}$\Rightarrow$\eqref{thm:Morita:it:std-eq} follows by considering the adjunction
    \[
    \begin{tikzcd}
    -\otimes_A P\colon\Mod{A}\rar[shift left]&\Mod{B}\noloc\Hom[B]{P}{-},\lar[shift left]
    \end{tikzcd}
    \]
    which is well defined since $A\cong\Hom[B]{P}{P}$ as algebras. Finally, the implication \eqref{thm:Morita:it:eq}$\Rightarrow$\eqref{thm:Morita:it:progen} follows from the following observation: An $A$-module $P$ is finitely generated and projective if and only if the $\kk$-linear functor
    \[
    \Hom[A]{P}{-}\colon\Mod{A}\longrightarrow\Mod{\kk}
    \]
    preserves all small colimits. Consequently, every equivalence between $\Mod{A}$ and $\Mod{B}$ restricts to an equivalence between $\proj{A}$ and $\proj{B}$, and we may take $P$ to be the image of $A$ under the latter.
\end{remark}

By definition, two algebras are \emph{Morita equivalent} if their module categories are equivalent as $k$-linear categories. 
At the risk of stating the obvious, \Cref{thm:Morita} reduces the question of when two algebras $A$ and $B$ are Morita equivalent to the problem of finding a progenerator for $B$ whose endomorphism algebra is isomorphic to $A$, which in principle seems like a more tractable problem.
The most basic example of Morita equivalent algebras is given by the ground commutative ring $\kk$ itself and the algebra of $n\times n$ matrices with entries in $\kk$, where $n\geq1$ (indeed, the latter algebra is the endomorphism algebra of the free $\kk$-module $\kk^n$). It is easy to show that Morita equivalence is preserved by the passage to the opposite algebras, hence this notion is left-right symmetric.\footnote{In general, we say that a property of an algebra $A$ is left-right symmetric when this property holds for $A$ if and only if it holds for $A^\op$.}

Morita equivalence can be detected on certain subcategories of the module category. Recall that an $A$-module $M$ is finitely presented if and only if the $k$-linear functor 
    \[
        \Hom[A]{M}{-}\colon\Mod{A}\longrightarrow\Mod{\kk}
    \]
    preserves filtered colimits.
It follows that every $\kk$-linear equivalence
\[
    \Mod{A}\stackrel{\sim}{\longrightarrow}\Mod{B}
\]
induces an equivalence
\[
    \mmod{A}\stackrel{\sim}{\longrightarrow}\mmod{B}.
\]
Recall also that $\mmod{A}$ is obtained from the category $\proj{A}$ by freely adjoining cokernels~\cite[III.2, Corollary]{Aus71}, and that $\Mod{A}$ is obtained from $\mmod{A}$ by freely adjoining filtered colimits, see for example~\cite[Exercise~6.8]{KS06}. This explains the following theorem:

\begin{theorem}
    \label{thm:Morita_subcategories}
    Let $A$ and $B$ be a pair of $\kk$-algebras. The following statements are equivalent.
    \begin{enumerate}
    \item There exists a $\kk$-linear equivalence of categories
        \[
            \Mod{A}\stackrel{\sim}{\longrightarrow}\Mod{B}.
        \]
    \item There exists a $\kk$-linear equivalence of categories
        \[
            \mmod{A}\stackrel{\sim}{\longrightarrow}\mmod{B}.
        \]
    \item There exists a $\kk$-linear equivalence of categories
        \[
            \proj{A}\stackrel{\sim}{\longrightarrow}\proj{B}.
        \]
    \end{enumerate}
\end{theorem}

    Morita equivalent algebras share many important invariants. For example, Morita equivalent algebras have isomorphic centres, since the centre of an algebra $A$ is isomorphic to the endomorphism algebra of the identity functor of the category $\Mod{A}$, see for example~\cite[Section~III.5.d]{Gab62}. Consequently, two commutative algebras are Morita equivalent if and only if they are isomorphic. Also, Morita equivalent algebras have isomorphic Grothendieck groups and, more generally, isomorphic higher algebraic $K$-theory groups, for these are defined in terms of the category of finitely generated projective modules~\cite[Sections~II.2 and~IV.6]{Wei13}.

Morita equivalences `propagate' under certain constructions. For example, Morita equivalence is preserved by the passage to categories of bimodules; this is a consequence of the following more precise version of the Eilenberg--Watts Theorem (see~\cite[Theorem~3.1]{NS16} for a proof of a more general statement): The functor
\[
    \Mod{A^\op\otimes_k B}\stackrel{\sim}{\longrightarrow}\LFun[\kk]{\Mod{A}}{\Mod{B}},\qquad M\longmapsto (-\otimes_A M),
\]
is an equivalence of $\kk$-linear categories, where the left-hand side is (equivalent to) the category of $A$-$B$-bimodules and the right-hand side denotes the category of cocontinuous $\kk$-linear functors ${\Mod{A}\to\Mod{B}}$. This equivalence can be used to produce many interesting examples of Morita equivalent algebras. Recall that, given an algebra $A$ and an $A$-bimodule $M$, one may form the split square-zero extension $A\ltimes M$, which is the algebra with underlying $\kk$-module $A\oplus M$ and the multiplication
\[
    (x,m)\cdot(x',m')\coloneqq (xx',xm'+x'm),\qquad x,x'\in A,\ m,m'\in M.
\]
It is not difficult to show that the category of $(A\ltimes M)$-modules can be reconstructed from the category $\Mod{A}$ and the endofunctor $-\otimes_A M\colon\Mod{A}\to\Mod{A}$, compare with~\cite[Lemma~10.11]{SY11a}. It follows that a $\kk$-linear equivalence
\[
\Phi\colon\Mod{A}\stackrel{\sim}{\longrightarrow}\Mod{B},
\]
induces a Morita equivalence between the algebras $A\ltimes M$ and $B\ltimes N$, where $N$ is the essentially unique $B$-bimodule such that there is natural isomorphism of $k$-linear functors
\[
    (-\otimes_B N)\cong\Phi(\Phi^{-1}(-)\otimes_A M).
\]

Similarly, working now over a field, the category of modules of the preprojective algebra $\Pi(A)$ of a hereditary finite-dimensional algebra $A$ (see~\cite{GP79,BGL87} for the definition) can be reconstructed from the category $\Mod{A}$ and the endofunctor
\[
-\otimes_A\Ext[A]{DA}{A}[1]\colon\Mod{A}\longrightarrow\Mod{A},\qquad DA\coloneqq\Hom[\kk]{A}{\kk},
\]
whose restriction to the category of finitely presented $A$-modules is the inverse Auslander--Reiten translation~\cite{Rin98}
\[
    \tau^{-}\colon\mmod{A}\longrightarrow\mmod{A}.
\]
As a consequence, Morita equivalent hereditary finite-dimensional algebras have Morita equivalent preprojective algebras.

Working again over an arbitrary commutative ring, suppose now that $A$ and $B$ are Morita equivalent algebras that are projective as $\kk$-modules. Then, the Hochschild cohomologies of $A$ and $B$ are isomorphic as graded algebras. Indeed, the Hochschild cohomology of $A$ can be identified with the graded algebra of self-extensions of the diagonal $A$-bimodule~\cite{CE99}
\[
    \HH{A}\coloneqq\Ext[A^e]{A}{A},\qquad A^e\coloneqq A\otimes_k A^\op,
\]
and the diagonal $A$-bimodule corresponds to the identity functor of $\Mod{A}$ under the Eilenberg--Watts equivalence. Although much more subtle, the shifted graded Lie algebra structure on Hochschild cohomology, given by the Gerstenhaber bracket~\cite{Ger63}, is also invariant under Morita equivalence, as can be deduced from Schwede's description of the Gerstenhaber bracket in terms of `loops of bimodule extensions' given in \cite{Sch98a} (see also~\cite{Kel04} and~\cite[Theorem~5.4.17]{Her16a}).

\subsection{Tilting theory: towards derived Morita Theory}

In this subsection we let $\kk$ be a field. Tilting theory has its origin in two results that we recall below; for more on the history of tilting theory and related developments we refer the reader to~\cite{AHHK07}.

In representation theory, the work of Bernstein, Gelfand and Ponomarev on reflection functors for quivers~\cite{BGP73} and its extension by Dlab and Ringel~\cite{DR76} to species in the sense of Gabriel~\cite{Gab73} led to the introduction of tilting modules by Brenner and Butler in~\cite{BB80}. Given a finite-dimensional algebra $B$, a \emph{tilting $B$-module} is a finite-dimensional $B$-module $T$ with the following properties:
\begin{itemize}
\item The projective dimension of $T$ is at most $1$ and
\[
    \Ext[B]{T}{T}[1]=0.
\]
In particular, $\Ext[B]{T}{-}[n]=0$ for all $n>1$.
\item  There is an exact sequence of $B$-modules
\[
    0\to B\to T^0\to T^1\to 0
\]
where $T^0$ and $T^1$ are direct summands of finite direct sums of copies of $T$.
\end{itemize}
The representation theory of the endomorphism algebra of $T$ exhibits the following relation to that of the original algebra $B$, which at the time was rather surprising:

\begin{theorem}[Brenner--Butler~\cite{BB80}]
\label{thm:BB}
    Let $B$ be a finite-dimensional algebra, $T$ a tilting $B$-module and set $A\coloneqq\End[B]{T}$. The following statements hold.
    \begin{enumerate}
    \item The following pair $(\T_B,\F_B)$ of subcategories of $\mmod{B}$ is a torsion pair in the sense of Dickson~\cite{Dic66}:
    \begin{align*}
        \T_B&\coloneqq\set{M\in\mmod{B}}[\Ext[B]{T}{M}[1]=0],\\
        \F_B&\coloneqq\set{M\in\mmod{B}}[\Hom[B]{T}{M}=0].
    \end{align*}
    Similarly, the following pair $(\T_A,\F_A)$ of subcategories of $\mmod{A}$ is a torsion pair:
    \begin{align*}
        \T_A&\coloneqq\set{N\in\mmod{A}}[N\otimes_A T=0],\\
        \F_A&\coloneqq\set{N\in\mmod{A}}[\operatorname{Tor}_1^{A}(N,T)=0].
    \end{align*}
    \item There are equivalences of categories
    \[
        \Hom[B]{T}{-}\colon \T_B\stackrel{\sim}{\longrightarrow}\F_A
    \]
    and
    \[
    \Ext[B]{T}{-}[1]\colon\F_B\stackrel{\sim}{\longrightarrow}\T_A.
    \]
    \end{enumerate}
\end{theorem}

On the other hand, in algebraic geometry, a famous theorem of Beilinson~\cite{Bei78} establishes an equivalence of derived categories
\[
    \DerCat[b]{\coh{\PP[n]}}\stackrel{\sim}{\longrightarrow}\DerCat[b]{\mmod{A}},
\]
between the bounded derived category of coherent sheaves on $n$-dimensional projective space and the bounded derived category of finite-dimensional modules over the finite-dimensional algebra $A\coloneqq\End[\PP[n]]{T}$, where
\[
\textstyle T\coloneqq\bigoplus_{i=0}^{n}\mathcal{O}(i)
\]
and $\mathcal{O}=\mathcal{O}_{\PP[n]}$ is the structure sheaf.
For example, if $n=1$ then $A$ is isomorphic to the path algebra of the Kronecker quiver $0\rightrightarrows 1$.  The coherent sheaf $T$ is the protoypical example of what would later be called a tilting object.

The parallel between the Brenner--Butler Theorem and Beilinson's Theorem is the following: Under the assumptions of \Cref{thm:BB}, Happel observed that there is an equivalence of bounded derived categories~\cite{Hap87}
\[
   \RHom[B]{T}{-}\colon\DerCat[b]{\mmod{B}}\stackrel{\sim}{\longrightarrow}\DerCat[b]{\mmod{A}}
\]  
that identifies $T$ with $A$.\footnote{Happel proved this when $A$ has finite global dimension although the result holds in general.} This equivalence is responsible for the equivalences in \Cref{thm:BB} for we have
\[
    H^\bullet(\RHom[B]{T}{M})\cong\Ext[B]{T}{M}[\bullet],\qquad M\in\mmod{B},
\]
see also~\cite[Sections~4.2 and~7.3]{Kel07}.

\begin{remark}
  Tilting modules of finite projective dimension were introduced by Miyashita in~\cite{Miy86}, who also established appropriate variants of \Cref{thm:BB} for these modules. Such generalised tilting modules were considered independently by Cline, Parshall and Scott in~\cite{CPS86}, who observed that these also induce equivalences of derived categories, thus extending the work of Happel. Another important novelty in these articles is that the authors work with arbitrary rings and not just with finite-dimensional algebras.
\end{remark}

\subsection{Derived Morita theory}

Recall that the \emph{derived category} of $A$ is, by definition, the localisation 
\[
\DerCat{\Mod{A}}\coloneqq\Ch{\Mod{A}}\![\mathrm{qis}^{-1}]
\]
of the category $\Ch{\Mod{A}}$ of unbounded complexes\footnote{Cohomological indexing convention is implicit, but it does not play an essential role in what follows.} of right $A$-modules at the class of quasi-isomorphisms. Although not immediately apparent from the definition, the category $\Mod{A}$ of $A$-modules embeds into the derived category as the full subcategory spanned by the complexes with cohomology concentrated in degree $0$ and therefore $\DerCat{\Mod{A}}$ can be regarded as an enlargement of the category of $A$-modules. Moreover, there are isomorphisms of $\kk$-modules
\[
    \Hom[\DerCat{\Mod{A}}]{M}{N[n]}\cong\Ext[A]{M}{N}[n],\qquad M,N\in\Mod{A},\ n\in\ZZ
\]
so that (equivalence classes of) extensions between $A$-modules can be interpreted as morphisms in the derived category.
However, unlike $\Mod{A}$, the derived category of $A$ is abelian if and only if the algebra $A$ is semi-simple. Nonetheless, the category $\DerCat{\Mod{A}}$ is always equipped with a distinguished class of sequences
\[
    X\to Y\to Z\to X[1],
\]
called \emph{exact triangles}, that satisfy certain axioms so that the derived category is a triangulated category in the sense of Verdier~\cite{Ver96}. Here, as is customary, we denote by $X[1]$ the complex obtained from the complex
\[
X\colon\quad\cdots\to X^{-1}\to X^0\to X^1\to\cdots
\]
by shifting it to the left and changing the sign of the differential. We also remind the reader that the exact triangles in a triangulated category are not characterised by universal properties---a source of substantial technical difficulties in the theory as this often prevents the use of certain categorical arguments in the study of derived categories, see for example~\Cref{rmk:technical_difficulties_I}.

Motivated by the analogous question for module categories as well as previous results in tilting theory, the following fundamental question, answered by Rickard's work, emerges: Given two algebras $A$ and $B$, when are the derived categories $\DerCat{\Mod{A}}$ and $\DerCat{\Mod{B}}$ equivalent as $\kk$-linear triangulated categories?

We need some preparation in order to state Rickard's main results.

\begin{notation}
Given an additive category $\C$, we write $\K{\C}$ for the \emph{homotopy category} of $\C$, that is the ideal quotient of the category of chain complexes in $\C$ by its ideal of null-homotopic morphisms. We also write $\K[+]{\C}$, $\K[-]{\C}$ and $\K[b]{\C}$ for the homotopy categories of complexes in $\C$ that are bounded on the left, bounded on the right, and bounded on both sides, respectively. If $\C$ is moreover abelian, we write $\DerCat{\C}$, $\DerCat[+]{\C}$, $\DerCat[-]{\C}$ and $\DerCat[b]{\C}$ for the derived categories of complexes in $\C$ that are unbounded, bounded on the left, bounded on the right, and bounded on both sides, respectively.
\end{notation}

We first need to introduce analogues of the finitely generated projective modules.

\begin{proposition}
    Let $A$ be an algebra. The canonical (exact) functor
    \[
    \K[b]{\proj{A}}\longrightarrow\DerCat{\Mod{A}}
    \]
    is fully faithful. By definition, its essential image consists of the \emph{perfect complexes} of $A$-modules.
\end{proposition}

\begin{remark}
\label{rmk:compacts-DA}
  A complex of $A$-modules $P$ is \emph{compact} if the
  $k$-linear functor
  \[
    \Hom[\DerCat{\Mod{A}}]{P}{-}\colon\DerCat{\Mod{A}}\longrightarrow\Mod{\kk}
  \]
  preserves small coproducts. It is easy to see that perfect complexes are
  compact, and the converse is also true (see~\Cref{thm:Neeman_compacts} below
  and compare also with the proof of~\cite[Proposition~6.3]{Ric89a}). Since being compact is a categorical property of an object, it follows that every (exact) $k$-linear equivalence of derived categories
  \[
    \DerCat{\Mod{A}}\stackrel{\sim}{\longrightarrow}\DerCat{\Mod{B}}
  \]
  restricts to an (exact) equivalence between $\K[b]{\proj{A}}$ and $\K[b]{\proj{B}}$.
\end{remark}

\begin{remark}
  More generally, given a regular cardinal $\alpha$, write $\DerCat{\Mod{A}}^\alpha$ for the full triangulated subcategory of $\DerCat{\Mod{A}}$ of \emph{$\alpha$-compact objects}, see ~\cite[Sec.~4.2]{Nee01} or~\cite{Kra01a} for the definition. For example, $\DerCat{\Mod{A}}^{\aleph_0}=\K[b]{\proj{A}}$. Every (exact) $k$-linear equivalence of derived categories
  \[
    \DerCat{\Mod{A}}\stackrel{\sim}{\longrightarrow}\DerCat{\Mod{B}}
  \]
  restricts to an (exact) equivalence between $\DerCat{\Mod{A}}^\alpha$ and $\DerCat{\Mod{B}}^\alpha$. Recall also that an $A$-module $M$ is \emph{$\alpha$-compact} if the functor
  \[
    \Hom[A]{M}{-}\colon\Mod{A}\longrightarrow\Mod{\kk}
  \]
  preserves $\alpha$-filtered colimits. Equivalently, $M$ is $\alpha$-compact if it admits a presentation by free $A$-modules of rank strictly less than $\alpha$. Write $\mmod[\alpha]{A}$ for the category of $\alpha$-compact $A$-modules. For example, $\mmod[\aleph_0]{A}=\mmod{A}$. Depending on $A$, there exists a regular cardinal $\alpha_0$ such that, for each regular cardinal $\alpha\geq\alpha_0$, the category $\mmod[\alpha]{A}$ is an abelian subcategory of $\Mod{A}$ and, moreover, there is an equivalence of triangulated categories 
  \[
    \DerCat{\mmod[\alpha]{A}}\stackrel{\sim}{\longrightarrow}\DerCat{\Mod{A}}^\alpha,
  \]
  see~\cite[Corollary~5.2 and Proposition~5.8]{Kra15a}.
\end{remark}

Let $A$ be an algebra. The regular representation of $A$ is clearly an example of a perfect complex (concentrated in degree $0$) and it enjoys the following additional properties as an object of the derived category:
\begin{itemize}
\item There is an isomorphism of algebras 
\[
    A\cong\Hom[\DerCat{\Mod{A}}]{A}{A}.
\]
\item For each non-zero integer $n$, we have
\[
    \Hom[\DerCat{\Mod{A}}]{A}{A[n]}\cong\Ext[A]{A}{A}[n]=0.
\]
\item The regular representation of $A$ generates $\K[b]{\proj{A}}$ as an idempotent-complete triangulated category.
\end{itemize}
Thus, the existence of a $k$-linear equivalence of triangulated categories
\[
    \DerCat{\Mod{A}}\stackrel{\sim}{\longrightarrow}\DerCat{\Mod{B}}
\]
implies the existence of a complex $T\in\DerCat{\Mod{B}}$ with analogous properties to those listed above. The following theorem of Rickard shows that the existence of such an object $T$ is not only necessary but also sufficient.

\begin{theorem}[Rickard's Derived Morita Theorem~{\cite{Ric89a,Ric91}}]
\label{thm:Rickards_Morita_Theorem_I}
 Let $A$ and $B$ be a pair of  $\kk$-algebras whose underlying $\kk$-modules are flat. The following statements are equivalent.
 \begin{enumerate}
     \item\label{thm:Rickards_Morita_Theorem_I:it:eq} There exists a $\kk$-linear equivalence of triangulated categories
     \[
        \DerCat{\Mod{A}}\stackrel{\sim}{\longrightarrow}\DerCat{\Mod{B}}.
     \]
     \item\label{thm:Rickards_Morita_Theorem_I:it:std-eq} There exists a complex of $A$-$B$-bimodules $M$ such that the $\kk$-linear functor
     \[
        -\Lotimes[A]M\colon\DerCat{\Mod{A}}\stackrel{\sim}{\longrightarrow}\DerCat{\Mod{B}}
    \]
    is an equivalence of triangulated categories. We call such a complex $M$ a \emph{two-sided tilting complex}.
    \item\label{thm:Rickards_Morita_Theorem_I:it:tilting_obj} There exists a perfect complex of $B$-modules $T$ with the following properties:
    \begin{itemize}
    \item There is an isomorphism of algebras
    \[
    A\cong\Hom[\DerCat{\Mod{B}}]{T}{T}.    
    \]
    \item  The perfect complex $T$ is a \emph{(one-sided) tilting complex}, that is:
    \begin{itemize}    
    \item For each non-zero integer $n$, we have
    \[
    \Hom[\DerCat{\Mod{B}}]{T}{T[n]}=0.    
    \]
    \item $T$ generates $\K[b]{\proj{B}}$ as an idempotent complete triangulated category. Equivalently, the regular representation of $B$ can be obtained from $T$ and its shifts by forming extensions and taking direct summands.\footnote{Despite the similarity with the notion of a progenerator, verifying that $T$ satisfies this criterion is in general much harder as one needs to deal with extensions that are not necessarily split.}
    \end{itemize}
    \end{itemize}
 \end{enumerate}
\end{theorem}

\begin{remark}
     The equivalence \eqref{thm:Rickards_Morita_Theorem_I:it:eq}$\Leftrightarrow$\eqref{thm:Rickards_Morita_Theorem_I:it:tilting_obj} does not require any flatness assumptions. This was shown by Rickard in~\cite{Ric89a} where instead of \eqref{thm:Rickards_Morita_Theorem_I:it:eq} he considered the following \emph{a priori} weaker condition:
     \begin{enumerate}
    \item[(1')]   There exists a $\kk$-linear equivalence of triangulated categories
     \[
        \DerCat[-]{\Mod{A}}\stackrel{\sim}{\longrightarrow}\DerCat[-]{\Mod{B}}.
     \]
     \end{enumerate}
     The extension to unbounded derived categories in the non-flat case was obtained by Keller in~\cite{Kel94}.
\end{remark}

\begin{remark}
\label{rmk:technical_difficulties_I}
In \Cref{thm:Rickards_Morita_Theorem_I}, it is clear that \eqref{thm:Rickards_Morita_Theorem_I:it:std-eq}$\Rightarrow$\eqref{thm:Rickards_Morita_Theorem_I:it:eq}$\Rightarrow$\eqref{thm:Rickards_Morita_Theorem_I:it:tilting_obj} without any flatness assumptions. The implication \eqref{thm:Rickards_Morita_Theorem_I:it:eq}$\Rightarrow$\eqref{thm:Rickards_Morita_Theorem_I:it:std-eq} would follow from a derived version of the Eilenberg--Watts Theorem, that is, a hypothetical triangulated equivalence between the derived category of complexes of $A$-$B$-bimodules and the category of $k$-linear exact functors $\DerCat{\Mod{A}}\to\DerCat{\Mod{B}}$ that preserve small coproducts. But the latter category of functors is not a triangulated category in general. Similarly, the implication \eqref{thm:Rickards_Morita_Theorem_I:it:tilting_obj}$\Rightarrow$\eqref{thm:Rickards_Morita_Theorem_I:it:std-eq} would follow if one was able to promote $T$ to a complex of $A$-$B$-bimodules in order to gain access to the derived adjunction
\[
    \begin{tikzcd}
        -\Lotimes[A]T\colon\DerCat{\Mod{A}}\rar[shift left]&\DerCat{\Mod{B}}\noloc\RHom[A]{T}{-}\lar[shift left],
    \end{tikzcd}
\]
but the endomorphism algebra of the complex $T\in\DerCat{\Mod{B}}$ does not act on its components.
Rickard's proof of \Cref{thm:Rickards_Morita_Theorem_I} is therefore technically involved.
On the other hand, these issues can be circumvented if instead of derived categories one works with their differential graded variants, see~\cite{Kel93,Kel94} and~\Cref{section:new_perspectives}. It is also worth noticing that the work of Spaltenstein~\cite{Spa88} on the existence of resolutions of unbounded complexes was essentially contemporaneous with Rickard's paper~\cite{Ric89a}.
\end{remark}

\begin{remark}
    It is not known whether every $\kk$-linear triangulated equivalence between derived categories of algebras is of \emph{standard type}, that is functorially isomorphic to the derived tensor product with a complex of bimodules, see~\cite{MY01,Che16,Min12,CZ19,BC24} among others for progress on this question. The analogous question in algebraic geometry is whether every exact $\kk$-linear functor between bounded derived categories of coherent sheaves on smooth projective schemes is of Fourier--Mukai type~\cite{BLL04}, see~\cite{Bal09,Ola24,Orl97,Toe07} for positive results and~\cite{RVdBN19,Kuen24} for counterexamples. As we shall see in \Cref{thm:dgEW}, in the differential graded setting this question has a positive answer.
\end{remark}

Two algebras are \emph{derived equivalent} if their derived categories are equivalent as $\kk$-linear triangulated categories. Clearly, Morita equivalence implies derived equivalence, but the converse is false in general: Let $\kk$ be a field and consider the path algebras $A=\kk(3\to 2\to 1)$ and $B=\kk(3\leftarrow 2\to 1)$. These algebras are not Morita equivalent, as can be seen by comparing the maximum Loewy lengths of the corresponding indecomposable projective modules. On the other hand, $B\cong\End[A]{T}$ is the endomorphism algebra of the tilting $A$-module~\cite{APR79}
\[
    T\coloneqq P_1\oplus P_2\oplus\tau^{-}(P_3)
\]
and, consequently, $A$ and $B$ are derived equivalent by Happel's Theorem~\cite{Hap87} and Rickard's \Cref{thm:Rickards_Morita_Theorem_II} below.

It is not difficult to show that, similar to Morita equivalence, the notion of derived equivalence is left-right symmetric~\cite[Proposition~9.1]{Ric89a}. One can then also ask which invariants are preserved under derived equivalence. A basic example is the Grothendieck group 
\[
    K_0(\K[b]{\proj{A}})\cong K_0(\proj{A})
\]
of the category of perfect complexes~\cite[Proposition~9.3]{Ric89a}. On the other hand, the global dimension of an algebra is not invariant under derived equivalence, but only its \emph{finiteness}~\cite{Hap88,GR97,KK20}. Rickard also establishes the following result. 

\begin{theorem}[{\cite{Ric89a,Ric91}}]
\label{thm:invariance-HH}
    Let $A$ and $B$ be a pair of $\kk$-algebras and suppose that there exists a $\kk$-linear equivalence of triangulated categories
    \[
    \DerCat{\Mod{A}}\stackrel{\sim}{\longrightarrow}\DerCat{\Mod{B}}.
    \]
    Then, $A$ and $B$ have isomorphic centres. Moreover, if the underlying $\kk$-modules of $A$ and $B$ are projective, then there is an isomorphism
    \[
        \HH{A}=\Ext[A^e]{A}{A}\stackrel{\sim}{\longrightarrow}\Ext[B^e]{B}{B}=\HH{B}
    \]
    between the corresponding Hochschild cohomology graded algebras.
\end{theorem}

\begin{remark}
    The shifted graded Lie algebra structure in Hochschild cohomology is also invariant under derived equivalence~\cite{Kel04}, see also~\cite{Kel03} for a more general result.
\end{remark}

\begin{remark}
    It follows from \Cref{thm:invariance-HH} that two commutative algebras are derived-equivalent if and only if they are isomorphic.
\end{remark}

One can also ask whether representation-theoretic properties of an algebra are invariant under derived equivalence. Recall that a finite-dimensional algebra $A$ over a field $\kk$ is \emph{symmetric} if there exists an isomorphism of $A$-bimodules $A\cong DA$ between the diagonal bimodule and its $\kk$-linear dual. The following result of Rickard shows that symmetric algebras are closed under derived equivalence.

\begin{theorem}[{\cite{Ric91}}]
\label{thm:symmetric}
    Let $A$ and $B$ be derived-equivalent finite-dimensional algebras over a field $\kk$. Then, $A$ is symmetric if and only if $B$ is symmetric.
\end{theorem}

\begin{remark}
    The proof of \Cref{thm:symmetric} relies on an analysis of the compatibility between the derived Nakayama functor
    \[
    -\Lotimes[A] DA\colon\DerCat{\Mod{A}}\longrightarrow\DerCat{\Mod{A}}
    \]
    and derived equivalences of standard type~\cite[Proposition~5.2]{Ric91}.
\end{remark}

\begin{remark}
    Another class of algebras that is closed under derived equivalence is that of gentle algebras~\cite{AH81,AS87}; this was shown by Schr{\"o}er and Zimmermann in~\cite{SZ03}. Gentle algebras are derived equivalent to partially wrapped Fukaya categories of surfaces~\cite{HKK17,LP19}, and the classification of gentle algebras up to derived equivalence was achieved recently in~\cite{APS23,Opp19} by verifying a conjecture from~\cite{LP19}.
\end{remark}

Continuing with the same theme, one may ask which constructions preserve derived equivalences. Rickard proves the following result in this direction.

\begin{theorem}[{\cite{Ric89b}}]
\label{thm:TA}
Let $A$ and $B$ be derived-equivalent finite-dimensional algebras over a field $\kk$.  Then, their trivial extensions $A\ltimes DA$ and $B\ltimes DB$ are derived equivalent.
\end{theorem}

\begin{remark}
    Rickard also proves a substantial generalisation of \Cref{thm:TA} for split square-zero extensions, see \cite[Corollary~5.4]{Ric91} for details. Notice that, due to the lack of functoriality of cones in a triangulated category, the derived category of a split square-zero extension cannot be reconstructed from the derived category of the algebra and the extending bimodule, at least if one wishes to avoid the use of higher structures.
\end{remark}

In the context of derived categories, Rickard proves the following variant of~\Cref{thm:Morita_subcategories}. Recall that an algebra $A$ is \emph{right coherent} if the category $\mmod{A}$ is closed under kernels in $\Mod{A}$, 

\begin{theorem}[{\cite{Ric89a,Ric91}}]
    \label{thm:Rickards_Morita_Theorem_II}
    Let $A$ and $B$ be two $\kk$-algebras whose underlying $\kk$-modules are flat. The following statements are equivalent:
    \begin{enumerate}
    \item There exists a $\kk$-linear equivalence of triangulated categories 
    \[
        \DerCat{\Mod{A}}\stackrel{\sim}{\longrightarrow}\DerCat{\Mod{B}}.
    \]
    \item There exists a $\kk$-linear equivalence of triangulated categories 
    \[
        \DerCat[-]{\Mod{A}}\stackrel{\sim}{\longrightarrow}\DerCat[-]{\Mod{B}}.
    \]
    \item There exists a $\kk$-linear equivalence of triangulated categories 
    \[
        \DerCat[b]{\Mod{A}}\stackrel{\sim}{\longrightarrow}\DerCat[b]{\Mod{B}}.
    \]
    \item\label{thm:Rickards_Morita_Theorem_II:it:modA}There exists a $\kk$-linear equivalence of triangulated categories 
    \[
        \K[b]{\proj{A}}\stackrel{\sim}{\longrightarrow}\K[b]{\proj{B}}.
    \]
\end{enumerate}
If $A$ and $B$ are right coherent, then the above are equivalent to the following additional statement:
\begin{enumerate}[resume]
    \item\label{thm:Rickards_Morita_Theorem_II:it:pcohA} There exists a $\kk$-linear equivalence of triangulated categories 
    \[
        \DerCat[b]{\mmod{A}}\stackrel{\sim}{\longrightarrow}\DerCat[b]{\mmod{B}}.
    \]
\end{enumerate}
\end{theorem}

\begin{remark}\label{rmk:pcoh}
    In \Cref{thm:Rickards_Morita_Theorem_II}, statement \eqref{thm:Rickards_Morita_Theorem_II:it:pcohA} can be replaced by the following, which does not require the right coherence assumption:
\begin{enumerate}
    \item[(5')] There exists a $\kk$-linear equivalence of triangulated categories 
    \[
        \DerCat[b]{\pcoh{A}}\stackrel{\sim}{\longrightarrow}\DerCat[b]{\pcoh{B}}.
    \]
\end{enumerate}
    Here, $\pcoh{A}$ denotes the exact category of \emph{pseudo-coherent} $A$-modules, that is those $A$-modules that admit a projective resolution with finitely generated components, and similarly for $\pcoh{B}$. Clearly, $\pcoh{A}=\mmod{A}$ whenever $A$ is right coherent. We refer the reader to~\cite{Nee90} for the definition of the bounded derived category of an exact category.
\end{remark}

\begin{remark}
    The proof of \Cref{thm:Rickards_Morita_Theorem_II} is rather technical. It amounts to the question of how various subcategories of the derived category of an algebra are determined by one another. Perhaps the most crucial point is to understand how $\K[b]{\proj{A}}$ determines the bounded derived category $\DerCat[b]{\pcoh{A}}$ of pseudo-coherent modules (\Cref{rmk:pcoh}). In \Cref{sec:metrics} we indicate how Neeman answers this question using metrics and completions of triangulated categories.
\end{remark}

\subsection{Modular representation theory and splendid equivalences}

The aim of this section is to relate and place into context the main ideas of Rickard's work in \cite{Rickard96} which, building on \cite{Ric89a} and \cite{Ric91}, introduced a special type of derived equivalence for blocks of group algebras in the context of Brou\'e's abelian defect group conjecture. 
Namely, Rickard introduced the notion of a \emph{splendid equivalence}  to structurally explain the phenomenon of Brou\'e's notion of isotypy of characters of finite groups, which are compatible families of perfect isometries. Indeed, the main contributions of Rickard's paper are the definition of splendid equivalences and the proof that a splendid equivalence between principal blocks of group algebras induces an isotypy at the level of their character groups. 

The representation theory of finite groups falls into two sorts: either the characteristic of the underlying field $k$ divides the order of the group or not. In the latter case, by Maschke's Theorem, every representation of the group decomposes into a direct sum of irreducible representations or equivalently a group algebra $kG$ is semisimple. 
Otherwise, this is no longer true and  the representation theory in this case,  which is commonly referred to as the \emph{modular representation theory of finite groups}, is much more complicated and many open questions remain. 
For example,  while for symmetric groups over the complex numbers everything is well understood, over a field of characteristic $p$ 
even the dimensions of the irreducible representations are not known in general. 

Two important approaches to tackle these open problems are to work over a so-called \emph{p-modular system} consisting  of a complete discrete valuation ring $\calo$ with maximal ideal $\mathfrak{m}$ such that the field of fractions $K$ of $\calo$ is a field of characteristic zero which is a splitting field for $G$ and all of its subgroups and the residue field $k$ of $\calo$ has characteristic $p >0$.  The structure of the group algebra $\mathcal{O} G$ is then closely related to both the structure of the group algebra $kG$ over the residue field and also to the structure of the semi-simple group algebra $KG$ over the field of characteristic zero. Similarly, there are deep connections between the categories of modules over $KG, \mathcal{O} G$ and $kG$, respectively.
A second approach is given by a fundamental program of research on global-to-($p$-)local  statements relating the representations or characters of a group to that of its  normalisers or centralisers of $p$-subgroups. Many of the most important conjectures such as the McKay conjecture \cite{McKay71} (the proof of which was recently completed in  \cite{CabanesSpaeth24}), Alperin's weight conjecture \cite{Alperin87},  and Brou\'e's  abelian defect group conjecture  \cite{Broue90} are in this global-to-local context and aim to provide reduction steps to smaller groups with an easier structure. 

The group algebra $kG$ of a finite group $G$ is decomposed into blocks  by the unique decomposition of its multiplicative identity  into a sum $e_1 + e_2 + \cdots + e_n$ of primitive orthogonal central idempotents.  So a  block $kG.e_i$ is an algebra with identity $e_i$ and is an indecomposable two-sided ideal of $kG$. This decomposition lifts uniquely to $\calo G$ inducing a natural bijection between the blocks of $kG$ and those of $\calo G$. Amongst the blocks, the \emph{principal block} is the one which does not annihilate the trivial module $k$ on which $G$ acts as the identity. 
 
 Each block $kG.e$ has an associated $p$-subgroup $P\subseteq G$, called the \textit{defect group} of the block, which measures how far the block is from being semisimple. More precisely, $P$ is a maximal $p$-subgroup such that $kP$ is isomorphic to a direct summand of $kG.e$ as a $kP$-$kP$-bimodule. Equivalently, $P$ is the minimal $p$-subgroup such that every $kG$-module is relatively $P$-projective, that is a direct summand of an induced module $kP\otimes_{kG} M$ for some $kP$-module $M$. We note that this defines $P$ uniquely up to conjugation in $G$. 
In the case of the principal block the defect groups are the Sylow $p$-subgroups and so its structure can be seen to be as at least as complicated as that of any other block.

The following result of Brauer \cite{Brauer56} can be seen as one of the foundational steps of the global-to-local programme. Namely, he showed that there is a bijective correspondence between blocks of $kG$ with defect group $D$ and blocks of $kN_G(D)$ with defect group $D$. The block $kN_G(D).f$ corresponding to a block $kG.e$ is called the \textit{Brauer correspondent of $kG.e$}.

The study of the representation theory of finite groups has its origin in their character theory over $K$. Observations stemming from properties or correspondences of characters and their deeper structural significance, for example in terms of equivalences of the associated module or derived categories, give rise to much of the contemporary work in modular representation theory.  




As an example, given finite groups $G$ and $H$, any derived equivalence between blocks
$KG.e$ and $KH.f$ is given by a  correspondence (with signs) of irreducible characters, that is in fact, an isometry with respect to the natural inner product of characters. However, passing to derived equivalences over $\calo$,  the corresponding isometries have extra properties, giving rise to so-called \emph{perfect isometries}. Namely, the following theorem by Brou\'e can be taken as a definition of  a perfect isometry.

\begin{theorem}[{\cite{Broue90}}]
For finite groups $G$ and $H$, a derived equivalence between blocks $\calo G.e$ and $\calo H.f$ induces an isometry between the (virtual) characters of $KG.e$ and $KH.f$ which preserves the subgroups of characters spanned by characters of projective modules for $\calo G.e$ and $\calo H.f$.
\end{theorem}

Note that a perfect isometry can also be characterised by purely arithmetic properties of characters which can be `read off' the character table of the groups. 

Since a derived equivalence over $\calo$ induces  a derived equivalence over $k$, we obtain as a consequence that the corresponding blocks $kG.e$ and $kH.f$ have the same number of simple modules. 

Brou\'e put forward the idea that a perfect isometry  should be a shadow of a structural correspondence at the level of the corresponding derived categories, giving rise to  Brou\'e's abelian defect group conjecture \cite{Broue90}: 

\begin{conjecture}[Brou\'e's Abelian Defect Group Conjecture] \label{conj:Brouederived}
Let $G$ be a finite group and let $\calo G. e$ be a block of $\calo G$ with abelian defect group given by the $p$-group $D$. Then, there is an equivalence of triangulated categories
$$\operatorname{D}^{b}(\mmod{\calo G.e}) \simeq \operatorname{D}^{b}(\mmod{\calo N_G(D).f}),$$
where $\calo N_G(D).f$ is the Brauer correspondent of $\calo G.e$ with respect to $D$. 
\end{conjecture}

The Brauer correspondent of a block is defined by the \emph{Brauer construction}, which also plays a role in the definition of Rickard's splendid equivalences. 
Namely, let $Q$ be a $p$-subgroup of a finite group $G$. The Brauer construction with respect to $Q$ is the functor 
$$-(Q)\colon \mmod{kG} \longrightarrow \mmod{kN_G(Q)}$$
given by $$\textstyle M(Q) \coloneqq M^Q / (\sum_{Q' <Q} {\rm Tr}^Q_{Q'}(M^{Q'}),$$
where for a finite group $H$ acting on $M$, $M^H$ denotes the $H$-fixed points of $M$ and ${\rm Tr}^Q_{Q'}(M^{Q'})$ are the relative traces of the ${Q'}$-fixed points of $M$ where ${Q'}$ is a proper subgroup of $Q$,  that is ${\rm Tr}^Q_{Q'}\colon  M^{Q'} \to M^Q $  is the map sending $m \in M^{Q'}$ to $ \sum_{x{Q'} \in Q/{Q'}} xm$. Note that if $M = k[X]$ is a permutation $kG$-module, then $M(Q)$ is the permutation $kN_G(Q)$- module $k[X^Q]$ on the fixed points of $Q$ acting on $X$. (If $Q$ is not a $p$-subgroup then there is no such functor on permutation modules.)  Furthermore, if $M$ is a $k$-algebra and $G$ acts by $k$-algebra automorphisms, then $M(Q)$ is also a $k$-algebra (with action of $N_G(Q)$ by automorphisms).  Thus the Brauer correspondent of a block $kG.e$ with defect group $D$ is the unique block $kN_G(D).f$ having defect group $D$ with the property that $kG.e(\Delta D)$ and $kN_G(D).f(\Delta D)$ are isomorphic as $k(C_G(D) \times C_G(D))$-modules. In particular, 
the Brauer correspondent 
of the principal block 
of $kG$
is the principal block 
of $kN_G(D)$.

In order to relate  Brou\'e's abelian defect group conjecture to  Rickard's splendid equivalences, we need the notion of an \emph{isotypy}, a strengthening of a perfect isometry. For this we say that two groups $G$ and $H$ with a common Sylow $p$-subgroup $P$ have the same \emph{$p$-local structure}  if for any subgroups $Q$ and $Q'$ of $P$ and an isomorphism $\theta: Q \to Q'$, there is an element $g \in G$ such that, for all $q \in Q$, $\theta(q) = g^{-1}qg$ if and only if there is an element $h \in H$ such that, for all $q \in Q$, $\theta(q) = h^{-1}qh$. Note that this is always true when $P$ is abelian and $H$ contains $N_G(P)$, which is precisely the setup for Brou\'e's abelian defect group conjecture.

Brou\'e then further conjectured that, for $G$ and $H$  finite groups with a common abelian Sylow $p$-subgroup and common \emph{$p$-local structure},   there should not only be a perfect isometry,  but an \emph{isotypy}, that is not only a character correspondence on the level of the principal blocks of  $G$ and $H$, but rather that there should be  a family of induced and compatible perfect isometries between the characters of the principal blocks of $ C_G(Q)$ and $ C_H(Q)$, for all $p$-subgroups $Q$ of $D$, see \cite{Broue90} for the precise definition. 

Splendid equivalences were then  introduced by Rickard  in \cite{Rickard96} to give a structural explanation of Brou\'e's isotypies for principal blocks. Later Harris \cite{Harris99} and Linckelmann \cite{Linckelmann98}, see also Puig \cite{Puig99}, generalised Rickard's notion of splendid equivalence to arbitrary $p$-blocks.  As such, splendid equivalences are derived equivalences between blocks of group algebras with the same $p$-local structure and which induce splendid equivalences at the $p$-local level. The name splendid equivalence derives from the fact that a splendid equivalence is induced by \textbf{SPL}it \textbf{END}omorphism two-sided tilting complexes of summands of permutation modules \textbf{I}nduced from \textbf{D}iagonal subgroups.  


We now explain in more detail what a splendid equivalence is. For this, we denote by $R$ either $\calo$ or $k$ if the statements hold in both cases. 

\begin{definition}[{\cite{Rickard96}}]\label{def:splendid}
Let $A$ and $B$ be symmetric $R$-algebras which are projective over $R$.  A bounded complex of finitely generated $A$-$B$-bimodules is a \emph{split-endomorphism two-sided tilting complex} - abbreviated SPLEND tilting complex in the following - if it is a complex of terms which are projective as both left and as right modules and
$${\rm Hom}_A(X,X) \cong {\rm Hom}_R(X,R) \otimes_A X \cong B $$
and 
$${\rm Hom}_B(X,X) \cong  X \otimes_B {\rm Hom}_R(X,R) \cong A $$
in the homotopy categories of complexes of $A$-bimodules and $B$-bimodules, respectively. 
\end{definition}

 Recall from Rickard's derived Morita Theory, that there is an equivalence between $D^b(\mmod A) \simeq D^b(\mmod B)$ of triangulated categories if and only if there exists a two-sided tilting complex of $A$-$B$-bimodules. Rickard notes that, under the hypothesis of Definition~\ref{def:splendid} on $A$ and $B$, any two-sided tilting complex of $A$-$B$-bimodules is quasi-isomorphic to a SPLEND two-sided tilting complex.

 The last step to obtain a SPLENDID tilting complex from a SPLEND tilting complex is to ensure good behaviour when applying the Brauer construction in the setting of block algebras of finite groups.   

\begin{definition}[{\cite{Rickard96}}]
    Let $G$ and $H$ be finite groups such that there is a common Sylow $p$-subgroup $P$ and let $RG.e$ and $RH.f$ be blocks of $G$ and $H$, respectively. A \emph{splendid tilting complex} $X$ is a complex of finitely generated $RG.e$-$RH.f$-bimodules such that 
\begin{itemize}
    \item[(i)] $X$ is a split-endomorphism tilting complex; 
    \item[(ii)] considering  $X$ as a complex of $R[G \times H]$-modules, its terms are direct summands of relatively $\Delta P$-projective permutation modules, where $\Delta P$ is the diagonal subgroup $\{(p,p) \in G \times H| p \in P\}$ of $G \times H$.
\end{itemize}    
\end{definition}



Two of the main results of \cite{Rickard96} in terms of principal blocks
are now the following. 

\begin{theorem}[{\cite[Theorem 4.1]{Rickard96}}]\label{thm:splendidBrauer}
    Let $G$ and $H$ be two finite groups that have a common Sylow $p$-subgroup $P$ and the same $p$-local structure and let $kG.e$ and $kH.f$ be the principal blocks of  $kG$ and $kH$, respectively and  $X$ a splendid tilting complex of $kG.e$-$kH.f$-bimodules. 
    Then, applying the Brauer construction with respect to the diagonal subgroup $\Delta Q$ of $G \times H$, the complex $X(\Delta Q)$ is a splendid tilting complex of $kC_G(Q).e(\Delta Q)$-$kC_H(Q).f(\Delta Q)$-bimodules for all subgroups $Q$ of $P$ where $kC_G(Q).e(\Delta Q)$ and $kC_H(Q).f(\Delta Q)$ denote the respective principal blocks.
\end{theorem}

Moreover, Rickard showed that the splendid equivalences over $k$ in Theorem \ref{thm:splendidBrauer} lift uniquely to  splendid equivalences over $\calo$. Furthermore, the above induces a splendid equivalence between the principal blocks of $\calo C_G(Q)$ and $\calo C_H(Q)$, for each subgroup $Q$ of $P$.  With this, the connection of splendid equivalences and isoptypies of characters is shown.

\begin{theorem}[{\cite[Theorem 6.3]{Rickard96}}]\label{thm:splendidinducesisotpy}
 Let $G$ and $H$ be two finite groups with a common Sylow $p$-subgroup and the same $p$-local structure. Then a splendid equivalence between the principal blocks of $\calo G$ and $\calo H$ induces an isotypy between these blocks. 
\end{theorem}

Motivated by Theorems~\ref{thm:splendidBrauer} and~\ref{thm:splendidinducesisotpy},  and based on subsequent work of Harris~\cite{Harris99}, Linckelmann~\cite{Linckelmann98} and Puig~\cite{Puig99} on non-principal blocks, Brou\'e's abelian defect group conjecture can be strengthened to

\begin{conjecture}\label{conj:Brouesplendid}
Let $G$ be a finite group and $\calo G.e$ a block of $G$ with an abelian $p$-group $D$ as its defect group. Then, there is a splendid equivalence between $\calo G.e$ and its Brauer correspondent $\calo N_G(D).f$  with respect to $D$.       
\end{conjecture}

  Brou\'e's original abelian defect group conjecture (Conjecture~\ref{conj:Brouederived}), remains open in general. However, it is noteworthy, that in cases where Conjecture~\ref{conj:Brouederived} has been verified, it is known that the refinement, Conjecture~\ref{conj:Brouesplendid}, also holds, even if it is not a formal consequence. 
There have been proofs, some of which are quite spectacular, in several special cases. For example, to name but a few,  Rickard himself showed the conjecture to hold for the principal block of the alternating group $\mathfrak A_5$ in \cite{Rickard96}.  In \cite{Rouquier98} the case of  cyclic defect groups has been shown. The case of  symmetric and general linear groups in non-defining characteristic has been shown in  \cite{ChuangRouquier08}  with a key step in the proof being the proof in \cite{ChuangKessar02} for the so-called Rock blocks of symmetric groups. Indeed, Chuang and Rouquier proved that blocks of symmetric groups with the same defect are derived equivalent which together with the previous work by Chuang and Kessar on RoCK blocks yielded the conjecture in the case of symmetric groups.  For double covers of symmetric groups at odd primes, very recently Conjecture~\ref{conj:Brouesplendid} has been shown in \cite{BrundanKleshchev25}, see also \cite{EbertLauraVera23},  by constructing derived equivalences from odd categorifications of $\mathfrak{sl}_2$ and building on the results in \cite{KleshchevLivesey25} in which the analogue of the Chuang-Kessar result on RoCK blocks has been shown for the RoCK blocks of double covers of symmetric groups. In other recent work, in \cite{DudasVaragnoloVasserot19}  building on \cite{Livesey15}, Conjecture~\ref{conj:Brouederived} has been shown in the case of unipotent blocks  of unitary groups for linear primes, the complimentary case of unitary primes still being open. However, whether Conjecture~\ref{conj:Brouesplendid} holds in this case  is still open. Moreover, in the context of finite reductive groups, Brou\'e's abelian defect group conjecture is closely linked with the underlying algebraic geometry, see \cite{BroueMalle93} and a proof of another related conjecture of Brou\'e's \cite{Broue90Red}, establishing a splendid equivalence between (sums of) blocks of finite reductive groups in non-defining characteristic  \cite{BonnafeRouquier03, BonnafeDatRouquier17}.

\section{New perspectives on derived Morita theory}
\label{section:new_perspectives}

In this section we discuss some further categorical aspects of derived Morita theory which arise in particular from the triangulated structure of the derived category of an algebra. One aspect of Rickard's theory is that there are several versions of the derived category for a given algebra, and it turns out that one can construct one from another via a notion of completion. Also, we look at further classes of triangulated categories beyond the derived
categories of associative algebras, where a Morita theory has been developed.

\subsection{Quillen's exact categories}

Recall that an \emph{exact category} (in the sense of Quillen~\cite{Qui73}) is an additive category equipped with a class of \emph{admissible} short exact sequences, also called \emph{conflations} satisfying suitable axioms, see~\cite{Bue10} for a survey of the theory. An extension-closed subcategory of an abelian category inherits the structure of an exact category with the short exact sequences contained in it as the admissible short exact sequences. For example, for a $k$-algebra $A$, the admissible short exact sequences in $\proj{A}\subseteq\Mod{A}$ are the split short exact sequences. Conversely, set-theoretic issues aside, every exact category is equivalent to an extension-closed subcategory of an abelian category with its induced exact structure~\cite[Proposition~A.2]{Kel90}. Exact categories provide a convenient enlargment of the class of abelian categories in which one can still do homological algebra. We refer the reader to~\cite{Nee90} for the definition of the bounded derived category of an exact category, which is used in what follows.

\subsection{Completions of triangulated categories}
\label{sec:metrics}

Let us return to Rickard's Derived Morita Theorem. In this section we
indicate how completions of triangulated categories arise when proving
some parts of this. We assume that $\kk$ is a commutative ring and observe that no flatness assumptions on the algebras are needed.

\begin{theorem}
    \label{thm:Rickards_Morita_Theorem_III}
    Let $A$ and $B$ be a pair of $\kk$-algebras. The following statements are equivalent:
    \begin{enumerate}
    \item There exists a $\kk$-linear equivalence of triangulated categories 
    \[
        \DerCat{\Mod{A}}\stackrel{\sim}{\longrightarrow}\DerCat{\Mod{B}}.
    \]
    \item There exists a $\kk$-linear equivalence of triangulated categories 
    \[
        \DerCat[b]{\pcoh{A}}\stackrel{\sim}{\longrightarrow}\DerCat[b]{\pcoh{B}}.
    \]
    \item There exists a $\kk$-linear equivalence of triangulated categories 
    \[
        \DerCat[b]{\proj{A}}\stackrel{\sim}{\longrightarrow}\DerCat[b]{\proj{B}}.
    \]
  \end{enumerate}
\end{theorem}

\begin{remark}
We view $\pcoh{A}$ and $\proj{A}$ as exact categories. When the ring $A$ is right coherent, then the category $\pcoh{A}$ equals $\mmod{A}$ and is an
abelian category. The canonical functor $\K[b]{\proj{A}}\to\DerCat[b]{\proj{A}}$ is an equivalence.
\end{remark}

\begin{remark}
  The chain of exact inclusions
  \[\proj{A}\hookrightarrow\pcoh{A}\hookrightarrow \Mod{A}\]
  induces fully faithful exact functors
    \[\DerCat[b]{\proj{A}}\rightarrow
      \DerCat[b]{\pcoh{A}}\rightarrow \DerCat{\Mod{A}}.\] The
    proof of Theorem~\ref{thm:Rickards_Morita_Theorem_III} shows that any of its
equivalences yields an equivalence when \emph{restricted} to one
of the subcategories. On the other hand, any of the above
equivalences \emph{extends} to an equivalence when one passes to
one of the bigger subcategories.
\end{remark}

Let us sketch the proof of the theorem, and further details will be provided afterwards. We need the following definition from \cite{Kra20}. A \emph{Cauchy sequence} in a triangulated category $\C$ is a sequence of morphisms
  \[X=( X_0\to X_1\to X_2\to\cdots)\] such that for all $C\in\C$ the
  induced map $\Hom{C}{X_i}\to \Hom{C}{X_{i+1}}$ is invertible for
  $i\gg 0$. In that case $\Hom{C}{X_i}=\colim\Hom{C}{X_j}$ for $i\gg 0$, and we say that $X$ is \emph{compactly supported} if for
  all $C\in\C$ we have $\colim\Hom{C}{\Sigma^n(X_j)}=0$ for $|n|\gg
  0$.

\begin{proof}[Proof of Theorem~\ref{thm:Rickards_Morita_Theorem_III}]
  (1)$\Rightarrow$(3) We have already seen that the canonical
  functor $\DerCat[b]{\proj{A}}\rightarrow \DerCat{\Mod{A}}$
  identifies $\DerCat[b]{\proj{A}}$ with the full subcategory of
  $\DerCat{\Mod{A}}$ consisting of the compact objects. Clearly, any
  equivalence preserves compactness.
    
  (3)$\Rightarrow$(1) We follow \cite{Kel94}. A quasi-inverse of
  $\DerCat[b]{\proj{A}}\stackrel{\sim}{\rightarrow}\DerCat[b]{\proj{B}}$
  identifies $B$ with a tilting object in $\DerCat[b]{\proj{A}}$,
  which we denote by $T$. We write $E:=\REnd[A]{T}$ for the differential
  graded endomorphism algebra of the complex $T$ and note that
  $H^\bullet(E)\cong B$. Viewing $A$ as a differential graded algebra,
  the functor $\RHom[A]{T}{-}$ yields an equivalence of triangulated categories
  $\DerCat{A}\stackrel{\sim}{\rightarrow}\DerCat{E}$ between the
  derived categories of differential graded modules, since $T$ is a
  tilting object. Now consider the pair of quasi-isomorphisms
  $B\twoheadleftarrow \tau^{\le 0}E\hookrightarrow E$, where
  $\tau^{\le 0}E$ denotes the soft truncation of $E$ concentrated in
  non-positive degrees. This yields the desired chain of equivalences
  \[\DerCat{\Mod{A}}=\DerCat{A}\stackrel{\sim}{\rightarrow}\DerCat{E}
    \stackrel{\sim}{\rightarrow}\DerCat{\Mod{B}}\] where the last
  functor is given by restriction along
  $\tau^{\le 0}E\hookrightarrow E$ composed with the left adjoint of
  restriction along $B\twoheadleftarrow \tau^{\le 0}E$. The
  construction of the functor shows that it sends $T$ to $B$.
  
  (1)$\Rightarrow$(2) The canonical functor $\DerCat[b]{\pcoh{A}}\rightarrow \DerCat{\Mod{A}}$ identifies
  the category $\DerCat[b]{\pcoh{A}}$ with the full subcategory of
  $\DerCat{\Mod{A}}$ consisting of the homotopy colimits of compactly
    supported Cauchy sequences in $\DerCat[b]{\proj{A}}$. It remains to observe that any
  equivalence preserves homotopy colimits and the notion of a compactly
  supported Cauchy sequence, given that $\DerCat[b]{\proj{A}}$ identifies with the full subcategory of compact objects in $\DerCat{\Mod{A}}$.

   (2)$\Rightarrow$(3) Proposition~\ref{pr:perfect} below provides an intrinsic
   description of the full subcategory consisting of the perfect
   complexes inside  $\DerCat[b]{\pcoh{A}}$ which uses nothing but
   the triangulated structure. Thus an equivalence of triangulated categories
   $\DerCat[b]{\pcoh{A}}\stackrel{\sim}{\rightarrow}\DerCat[b]{\pcoh{B}}$
   restricts to an  equivalence
   $\DerCat[b]{\proj{A}}\stackrel{\sim}{\rightarrow}\DerCat[b]{\proj{B}}$.
 \end{proof}

 The composite (3)$\Rightarrow$(1)$\Rightarrow$(2) yields the
 implication (3)$\Rightarrow$(2), which suggests to consider
 $\DerCat[b]{\pcoh{A}}$ as a completion of
 $\DerCat[b]{\proj{A}}$. This uses the following proposition which is
 based on a classical argument due to Milnor \cite{Mil62}.

We set $\C:=\DerCat[b]{\proj{A}}$ and write $\Mod{\C}$ for the
category of $\kk$-linear functors $\C^\op\to\Mod{\kk}$. 
   
  \begin{proposition}\label{pr:Milnor}
    The  functor $\DerCat{\Mod{A}}\to
    \Mod{\C}$ given by
    \[X\longmapsto h_X:=\Hom{-}{X}|_\C\] restricts to a fully faithful
    functor  \[\DerCat[b]{\pcoh{A}}\longrightarrow
      \Mod{\C}.\]
  \end{proposition}
  \begin{proof}
    See Proposition~5.2.8 in \cite{Kra22}.
  \end{proof}

  Any object in $\DerCat[b]{\pcoh{A}}$ is quasi-isomorphic to a
  bounded above complex of finitely generated projective
  $A$-modules, and using successive truncations it is not difficult to show that the
  objects in $\DerCat[b]{\pcoh{A}}$ are up to isomorphism precisely the homotopy colimits of
  compactly supported Cauchy sequences in $\DerCat[b]{\proj{A}}$; see
  \cite[Section~5.2]{Kra22} for details. In view of
  Proposition~\ref{pr:Milnor} it follows that $X\mapsto h_X$
  identifies $\DerCat[b]{\pcoh{A}}$ with the full subcategory of
  $\Mod{\C}$ consisting of the colimits of compactly supported Cauchy
  sequences in $\C$ (viewed as a full subcategory of $\Mod{\C}$ via
  the Yoneda embedding).\footnote{To be more precise: the homotopy colimit of a compactly supported Cauchy
  sequence in $\C$ taken in $\DerCat{\Mod{A}}$ is actually a colimit in $\DerCat[b]{\pcoh{A}}$.} The distinguished triangles in  $\DerCat[b]{\pcoh{A}}$
  are precisely the colimits of compactly supported Cauchy
  sequences of distinguished triangles in $\C$.

  The above procedure of completing the category of perfect complexes
  is discussed in \cite{Kra20}. However, Neeman proposes in
  \cite{Nee18,Nee20} a much more refined notion of `metric completions' for
  triangulated categories. More precisely, he shows that the
  completion of a triangulated category with respect to a \emph{good metric}
  yields a triangulated category. In this way $\DerCat[b]{\pcoh{A}}$
  is obtained from $\DerCat[b]{\proj{A}}$, but this works as well the
  other way round, so $\DerCat[b]{\proj{A}}$
  is obtained as a metric completion from $\DerCat[b]{\pcoh{A}}$.

  In the following we provide a fairly direct proof of (2)
  $\Rightarrow$ (3) that does not use any completion arguments. This is inspired by Neeman and taken from
  \cite[Section~9.2]{Kra22}, but here we remove the simplifying
  assumption that the ring $A$ is right coherent.  
  
 Set $\T:=\DerCat[b]{\pcoh{A}}$. For $n\in\mathbb Z$ set
\[\T^{> n}:=\{X\in\T\mid H^i(X)=0\textrm{ for all }i \le n\}\]
and
\[\T^{\le n}:=\{X\in\T\mid H^i(X)=0\textrm{ for all }i > n\}.\]
For an object $U\in\T$ and $n\in\mathbb Z$ set
\[\T_U^{> n}:=\{X\in\T\mid \Hom{U}{\Sigma^i(X)}=0\textrm{ for all }i \le n\}\]
and
\[\T_U^{\le n}:=\{X\in\T\mid \Hom{X}{\T_U^{> n}}=0\}.\]
Note that $\Hom{X}{Y}=0$ for all $X\in\T^{\le n}$ and
$Y\in\T^{>n}$. Also, we have $\T^{>n}=\T_A^{>n}$ since
$H^i(X)=\Hom{A}{\Sigma^i(X)}$ for all $X\in\T$ and $i\in\mathbb Z$.

\begin{lemma}
  If $U\in\T^{\le n}$, then $\T^{>n}\subseteq\T_U^{>0}$.
\end{lemma}
\begin{proof}
  The assumption $U\in\T^{\le n}$ implies $\Sigma^i(U)\in\T^{\le n}$ for all
  $i\ge 0$. This implies $\Hom{U}{\Sigma^i(X)}=0$ for all $i\le 0$ and
  $X\in\T^{>n}$. Therefore $\T^{>n}\subseteq\T_U^{>0}$.
\end{proof}

For objects $U,V\in\T$ write $U\le V$ if there are $p,q\in\mathbb Z$ such
that $\T_U^{>p}\subseteq \T_V^{>q}$. Call $U\in\T$
\emph{initial} if $U\leq V$ for
all $V\in\T$. The preceding lemma shows that $A$ is initial. Set
\[\T_U:=\{X\in\T\mid \operatorname{Hom}(X,\T_U^{\le n})=0\textrm{ for all }n\ll 0\}.\]

 The following result says that $\DerCat[b]{\proj{A}}$ admits an intrinsic description as a subcategory
 of $\T=\DerCat[b]{\pcoh{A}}$, since the notion of an initial object is purely categorical. 

\begin{proposition}\label{pr:perfect}
  Let $U\in\T$ be initial. Then $\DerCat[b]{\proj{A}}=\T_U$.
\end{proposition}
\begin{proof}
  Let $V\in\T$.  Then there are $p,q\in\mathbb Z$ such that
  $\T_U^{>p}\subseteq \T_V^{>q}$. This implies
  $\T_U^{\le p}\supseteq\T_V^{\le q}$, and therefore
  $\T_U\subseteq \T_V$. It follows that $\T_U=\T_V$ when $U$ and $V$
  are initial. Thus we may assume $U=A$. Clearly, $\T_U$ is a thick
  subcategory that contains $A$. Thus $\T_U$ contains
  $\DerCat[b]{\proj{A}}$. Now fix an object $X\in\T$ that is not in
  $\DerCat[b]{\proj{A}}$. We may replace $X$ by a bounded above
  complex of finitely generated projective $A$-modules and write
  $\sigma^{\le n}X$ for the brutal truncation at $n$.  Then the
  canonical morphism $X\to\sigma^{\le n}X$ is not null-homotopic for
  $n\ll 0$. Note that $\sigma^{\le n}X$ lies in $\T^{\le
    n}$. Thus $X\not\in\T_U$.
\end{proof}

\subsection{Morita theory for compactly generated triangulated categories}

Further study of the properties of derived categories of algebras and similar triangulated categories brings the following class of triangulated categories to the fore. In what follows, we denote the suspension functor of an abstract triangulated category by $\Sigma$.

\begin{definition}
    Let $\T$ be a triangulated category and suppose that $\T$ admits small coproducts. We say that $\T$ is \emph{compactly generated} if there exists a set $\G\subseteq\T$ consisting of compact objects and which generate $\T$ as a triangulated category with small coproducts.
\end{definition}

The derived category of an algebra $A$ is the prototypical example of a compactly generated triangulated category, with the regular representation of $A$ serving as a compact generator. Recall also that we claimed in \Cref{rmk:compacts-DA} that $\K[b]{\proj{A}}$ consists precisely of the compact objects in $\DerCat{A}$. This is a special case of the following general theorem.

\begin{theorem}[Neeman~{\cite{Nee92}}]
\label{thm:Neeman_compacts}
    Let $\T$ be a compactly generated triangulated category and $\G\subseteq\T$ a set of compact generators. Then, the full subcategory of $\T$ spanned by the compact objects coincides with its smallest idempotent complete triangulated subcategory which contains $\G$.
\end{theorem}

The following question is therefore natural: When are two compactly-generated triangulated categories $\S$ and $\T$ equivalent? Clearly, an equivalence
\[
    \S\stackrel{\sim}{\longrightarrow}\T
\]
induces an equivalence of triangulated categories between the corresponding subcategories of compact objects. However, additional structure seems to be necessary to prove a more general version of \Cref{thm:Rickards_Morita_Theorem_I}: to an object $X\in\T$ one may associate its graded endomorphism algebra
\[
    \textstyle\T(X,X)^\bullet\coloneqq\bigoplus_{i\in\ZZ}\T(X,\Sigma^i(X)),\qquad g*f\coloneqq \Sigma^j(g)\circ f,\ |f|=j,
\]
but the graded endomorphism algebra of a compact generator, say, need not determine $\T$ as a triangulated category.

Recall that a differential graded (dg, for short) category is a category enriched in complexes of $\kk$-modules. For example, dg algebras identify with dg categories with a single object. If $\A$ is a small dg category, then its derived category $\DerCat{\A}$ is a compactly generated triangulated category~\cite{Kel94}. The derived category of $\A$ has a salient property: it is equivalent to the stable category of a $\kk$-linear Frobenius exact category.\footnote{A Frobenius exact category is an exact category in the sense of Quillen that has enough projectives, enough injectives, and such that the classes of projectives and of injective objects coincide. Stable categories of Frobenius exact categories are triangulated by a well-known theorem of Happel~\cite{Hap88}, see also~\cite{Hel68}.} A triangulated category with the latter property is said to be \emph{algebraic}. The following theorem shows that all algebraic compactly generated triangulated categories are of this form; for missing definitions and more information on the theory of dg categories we refer the reader to~\cite{Kel06}.

\begin{theorem}[Keller's Recognition Theorem~{\cite{Kel94}}]
\label{thm:Kellers_Recognition}
    Let $\T$ be an algebraic compactly generated triangulated category. Then, there exists a small differential graded category and a $\kk$-linear equivalence of triangulated categories
    \[
        \T\stackrel{\sim}{\longrightarrow}\DerCat{\A}.
    \]
\end{theorem}

\begin{remark}
\label{rmk:dga-not-unique}
   In the context of \Cref{thm:Kellers_Recognition}, the differential graded category $\A$ in general is not determined up to quasi-equivalence by $\T$ even after fixing a preferred set of compact generators of the  (but see~\cite{LO10,CS17,CS18,CNS22} for important results of this kind). When $\A$ has a single object, this problem is related to the classification of differential graded algebras with a given cohomology graded algebra up to quasi-isomorphism: A $\kk$-linear equivalence of triangulated categories
    \[
    \T\stackrel{\sim}{\longrightarrow}\DerCat{A},
    \]
    which identifies $G$ and $A$ and
    induces an isomorphism of graded algebras
    \[
    \textstyle\bigoplus_{i\in\ZZ}\T(G,\Sigma^i(G))\cong H^\bullet(A).
    \]
   Nonetheless, the derived Morita theory of small dg categories is completely understood thanks to the results in~\cite{Kel94}, see~\cite[Theorem~3.11]{Kel06} for a precise formulation of the analogue of \Cref{thm:Rickards_Morita_Theorem_I} in this context. Similar to \Cref{thm:invariance-HH}, it is also known that the Hochschild cohomology of a small dg category is a derived Morita invariant; other derived Morita invariants include algebraic $K$-theory, Hochschild homology, cyclic homology, etc. We refer the reader to~\cite[Section~5]{Kel06} for details and references to the literature.
\end{remark}

\begin{remark}
    A typical application of \Cref{thm:Kellers_Recognition} is in the context of Kontsevich's Homological Mirror Symmetry Conjecture~\cite{Kon98}, which predicts an equivalence between two algebraic triangulated categories (a suitable flavour of a Fukaya category and a bounded derived category of coherent sheaves on an algebro-geometric object or a variant thereof). We refer the reader to \cite{PZ98,AKO06,Sei15} for some representative instances of this paradigm.
\end{remark}

Let $\A$ be a small differential graded category. The derived category $\DerCat{\A}$ admits a canonical differential graded enhancement in the sense of~\cite{BK90}, that is a pre-triangulated differential graded category $\dgDerCat{\A}$ and a $\kk$-linear equivalence of triangulated categories
\[
    \DerCat{\A}\simeq H^0(\dgDerCat{A}).
\]
We say that a pair of small dg categories are \emph{derived Morita equivalent} if their derived dg categories are quasi-equivalent. The following theorem illustrates the advantages of working with dg enhancements of triangulated categories (see also~\cite{LRG22} where the more general case of well-generated pre-triangulated dg categories is treated).

\begin{theorem}[To{\"e}n's Derived Eilenberg--Watts Theorem~{\cite{Toe07}}]
\label{thm:dgEW}
    Let $\A$ and $\B$ be small differential graded categories. Then, there is a canonical isomorphism
    \[
    \dgDerCat{\A^\op\Lotimes[k]\B}\stackrel{\sim}{\longrightarrow}\RHom[c]{\dgDerCat{\A}}{\dgDerCat{\B}}
    \]
    in the homotopy category of differential graded categories~\cite{Tab05a}, where the right-hand side is the differential graded category of quasi-functors $\dgDerCat{\A}\rightsquigarrow\dgDerCat{\B}$ whose induced $\kk$-linear exact functor
    \[
    \DerCat{\A}\longrightarrow\DerCat{\B}
    \]
    preserves small coproducts.
\end{theorem}

Rickard's \Cref{thm:symmetric,thm:TA} have the following generalisations for dg algebras (with many objects). Let $A$ be a dg algebra over a field $\kk$, and suppose that $A$ is locally proper, which is to say that its cohomology is degree-wise finite dimensional. Following Kontsevich~\cite{Kon93} (see also~\cite{KS06a}, we say that $A$ is \emph{bimodule right $n$-Calabi--Yau}, $n\in\ZZ$, if there exists an isomorphism $A[n]\cong DA$ in the derived category of dg $A$-bimodules. Notice that this property is invariant under derived Morita equivalence. The (bimodule right $n$-Calabi--Yau) cyclic completion $A\ltimes DA[1-n]$ in the sense of Segal~\cite{Seg08} (see also~\cite{Sei08a,Sei10}) gives an analogue of the passage to the trivial extension of a finite-dimensional algebra. These kinds of shifted trivial extensions play a central role in the construction of cluster categories of finite-dimensional algebras of finite global dimension~\cite{Kel05,Ami09,Guo11}, and also in the recent classification of Brauer algebras up to derived equivalence by Opper and Zvonareva~\cite{OZ22}.

The Morita invariance of preprojective algebras of hereditary finite-dimensional algebras has the following analogue for dg algebras. Still working over a field $\kk$, suppose now instead that the dg algebra $A$ is homologically smooth, that is the diagonal dg $A$-bimodule is compact. Following Kontsevich and Ginzburg~\cite{Gin06}, we say that $A$ is \emph{bimodule left $n$-Calabi--Yau}, $n\in\ZZ$, if there exists an isomorphism
\[
    \RHom[A^e]{A}{A^e}\cong A[-n]
\]
in the derived category of dg $A$-bimodules. Notice that this property is also invariant under derived Morita equivalence. For a homologically smooth dg algebra $A$, Keller's derived $n$-preprojective dg algebra $\mathbf{\Pi}_n(A)$ is bimodule left $n$-Calabi--Yau and the passage from a homologically smooth dg algebra to its derived $n$-preprojective algebra perserves derived Morita equivalences~\cite{Kel11}. This construction, which generalises the construction of the Ginzburg dg algebra of a quiver with potential~\cite{Gin06}, again plays a central role in the construction of cluster categories and also in Iyama's higher-dimensional version of Auslander--Reiten Theory~\cite{IO13,HIO14}. For applications in derived algebraic geometry see for example~\cite{BCS24}, and for applications in symplectic geometry see for example~\cite{EL17a}.

The previous results admit topological analogues. Following~\cite{SS03,Sch10a}, call a compactly generated triangulated category $\T$ \emph{topological} if it is equivalent, as a triangulated category, to the homotopy category of a pointed (closed) Quillen model category for which the suspension functor is an equivalence on its homotopy category. The most fundamental example is the homotopy category of spectra~\cite{BF78}, which is compactly generated by the sphere spectrum, see for example~\cite[Section~2.3]{SS03}.

\begin{theorem}[Schwede--Shipley's Recognition Theorem~\cite{SS03}]
\label{thm:SS_Recognition}
    Let $\T$ be a topological compactly generated triangulated category. Then, there exists a small spectral category $\A$, that is a category enriched in symmetric spectra~\cite{HSS00}, and an equivalence of triangulated categories
    \[
        \T\stackrel{\sim}{\longrightarrow}\DerCat{\A}
    \]
    between $\T$ and the derived category of $\A$.
\end{theorem}

\begin{remark}
    Similar to the algebraic case, the spectral category $\A$ in \Cref{thm:SS_Recognition} in general is not determined up to (Dwyer--Kan) equivalence by $\T$ and a preferred set of compact generators of the latter (but see~\cite{Sch01,SS02,Sch07} for important positive results of this kind). When $\A$ has a single object, this question is related to the problem of classifying the ring spectra with a given homotopy graded algebra: An equivalence of triangulated categories
    \[
        \T\stackrel{\sim}{\longrightarrow}\DerCat{A},
    \]
    which send $G$ to $A$ and induces an isomorphism of graded algebras
    \[
        \textstyle\bigoplus_{i\in\ZZ}\T(G,\Sigma^i(G))\cong\bigoplus_{i\in\ZZ}\pi_{-i}(A),
    \]
    compare with~\Cref{rmk:dga-not-unique}.
\end{remark}

\begin{remark}
    \Cref{thm:SS_Recognition} can also be formulated in terms of Lurie's stable $\infty$-categories~\cite[Theorem~7.1.2.1]{Lur17} and a version~\Cref{thm:dgEW} is also available in this context~\cite[Proposition~7.1.2.4]{Lur17}, see also~\cite{Hov15}.
\end{remark}

\begin{remark}
    For the sake of completeness, we mention that the homotopy theory of DG categories is equivalent to that of spectrally enriched categories that are linear over the Eilenberg--Mac Lane ring spectrum of the ground commutative ring~\cite{Mur15} (see also~\cite{Shi07}). Consequently, every algebraic triangulated category is topological. However, not every topological triangulated category is algebraic, as exemplified by the the stable homotopy category~\cite{Sch10a}, and distinguishing between topological and algebraic enhancements of a triangulated category is a subtle problem (see for example~\cite{DS07}).
\end{remark}

\printbibliography












\end{document}